\documentclass[12pt]{amsart}
\usepackage[margin=1.1in]{geometry}
\usepackage{pb-diagram}

\usepackage{amssymb,amsthm,amsmath,latexsym,stmaryrd}
\usepackage[all]{xy}

\newcommand{\complex}{{\mathbb C}}

\newcommand{\ball}{{\mathbb B}}

\newcommand{\calc}{{\mathcal C}}

\newcommand{\calm}{{\mathcal M}}

\newcommand{\calo}{{\mathcal O}}

\theoremstyle{plain}
        \newtheorem{theorem}{Theorem}[section]
        \newtheorem{lemma}[theorem]{Lemma}
        \newtheorem{remark}[theorem]{Remark}
        \newtheorem{proposition}[theorem]{Proposition}
        \newtheorem{corollary}[theorem]{Corollary}

        \newtheorem{question}[theorem]{Question}
\theoremstyle{definition}

\newtheorem{assumption}[theorem]{Assumption}

\title{An Analytic Grothendieck Riemann Roch Theorem}
\author{Ronald G. Douglas}
\address{Department of Mathematics, Texas A$\&$M University, College Station, TX, 77843}
\email{rdouglas@math.tamu.edu}

\author{Xiang Tang}
\address{Department of Mathematics, Washington University, St. Louis, Missouri, USA, 63130}
\email{xtang@math.wustl.edu}

\author{Guoliang Yu}
\address{Department of Mathematics, Texas  A$\&$M University, College Station, TX, 77843}
\email{guoliangyu@math.tamu.edu}

\begin{document}
\begin{abstract}We extend the Boutet de Monvel Toeplitz index theorem to complex manifold with isolated singularities following the relative $K$-homology theory of Baum, Douglas, and Taylor for  manifold with boundary. We apply this index theorem to study the Arveson-Douglas conjecture. Let $\ball^m$ be the unit ball in $\mathbb{C}^m$, and $I$ an ideal in the polynomial algebra $\mathbb{C}[z_1, \cdots, z_m]$. We prove that when the zero variety $Z_I$ is a complete intersection space with only isolated singularities and intersects with the unit sphere $\mathbb{S}^{2m-1}$ transversely, the representations of $\mathbb{C}[z_1, \cdots, z_m]$ on the closure of $I$ in $L^2_a(\ball^m)$ and also the corresponding quotient space $Q_I$ are essentially normal. Furthermore, we prove an index theorem for Toeplitz operators on $Q_I$ by showing that the representation of $\mathbb{C}[z_1, \cdots, z_m]$ on the quotient space $Q_I$ gives the fundamental class of the boundary $Z_I\cap \mathbb{S}^{2m-1}$.  In the appendix, we prove with Kai Wang that if $f\in L^2_a(\ball^m)$ vanishes on $Z_I\cap \ball ^m$, then $f$ is contained inside the closure of the ideal $I$ in $L^2_a(\ball^m)$. 
\end{abstract}
\maketitle
\section{Introduction}


Let $X$ and $Y$ be closed smooth complex manifolds, and $f: X\to Y$ be a proper smooth map. The classical Grothendieck-Riemann-Roch theorem \cite{bo-se:rr} relates the push forward maps on the Grothendieck groups $f_!: K_0(X)\to K_0(Y)$ of bounded complexes of coherent sheaves and the Chow groups $f_*: A(X)\to A(Y)$ of subvarieties modulo rational equivalence. More precisely, let $\operatorname{Ch}: K_0(X)\to A(X)$ be the Chern character map, and $\operatorname{Td}(X)\in A(X)$  be the Todd genus of $X$. The Grothendieck-Riemann-Roch theorem states that for a vector bundle $E$ on $X$, 
\[
\operatorname{Ch}\big(f_!(E)\big)\cup \operatorname{Td}(Y)=f_*\big( \operatorname{Ch}(E)\cup \operatorname{Td}(X)\big).
\]
In noncommutative geometry, the push forward map $f_*:K_
\bullet(X)\to K_\bullet(Y)$ on the $K$-homology group introduced by Brown, Douglas, and Fillmore \cite{bdf:k-hom} is related to the  push forward map $f_*: HP^\bullet(C^\infty(X))\to HP^\bullet(C^\infty(Y))$  on the periodic cyclic cohomology introduced by Connes \cite{c:cyclic} via the Connes-Chern character $\operatorname{Ch}: K_\bullet(-)\to HP^\bullet(-)$ satisfying 
\[
\operatorname{Ch}\big(f_*([D])\big)=f_*\big(\operatorname{Ch}([D])\big). 
\]
This can be viewed as a noncommutative generalization of the Grothendieck-Riemann-Roch theorem.



In this article, motivated from questions in operator theory, we are interested in extending the above study of push forward maps in two directions. 
\begin{enumerate}
\item Allow $X$ and $Y$ to have singularities.
\item Allow $X$ and $Y$ to have boundaries.   
\end{enumerate}
In the literature, many interesting works have been developed to address the above questions. For example, Baum, Fulton, and McPherson in \cite{bfm:srr-1}, \cite{bfm:srr-2} proved the Riemann-Roch theorem for a singular projective variety $X$; Baum and the first author in \cite{bd:theory}, \cite{bdt:boundary} studied  the relative $K$-homology groups, $K_\bullet(X, \partial X)$, for a manifold $X$ with boundary, $\partial X$. However, the formulation of the general Grothendieck-Riemann-Roch theorem, including the above two special cases, is missing. 

In this paper, we study the generalization of the Grothendieck-Riemann-Roch theorem in the following setting. Let $A=\complex[z_1, \cdots, z_m]$ be the polynomial ring of $m$ variables. Let $I$ be an ideal of $A$. Define 
\[
Z_I=\{(z_1, \cdots, z_m)\in \mathbb{C}^m\ :\ a(z_1, \cdots, z_m)=0, \forall a\in I \}. 
\] 
We point out that the analytic space $Z_I$ may have singularities. Let $\mathbb{B}^m$ be the open unit ball in $\complex^m$ with the natural euclidean metric, and $\partial \ball^m:=\overline{\ball}^m\backslash \ball^m$ be its boundary, the unit sphere $\mathbb{S}^{2m-1}$. Denote  $Z_I\cap \mathbb{B}^m$ by $\Omega_I$. The analytic space $\Omega_I$ is naturally a (singular) submanifold of $\mathbb{B}^m$ with the  boundary $\partial \Omega_I:=\overline{\Omega}_I\backslash \Omega_I=Z_I\cap \partial \ball^m$.

When $Z_I$ is smooth and intersects the unit sphere $\mathbb{S}^{2m-1}=\partial \ball^m$ transversely, $\overline{\Omega}_I$ is a smooth complex manifold with a strongly pseudoconvex boundary $\partial \Omega_I=Z_I\cap \mathbb{S}^m$ (of odd dimension). Baum, Douglas, and Taylor in \cite{bdt:boundary} introduced a $K$-homology class  $[D_N]$ in $K_0(\overline{\Omega}_I, \partial \Omega_I)$ from the $\bar{\partial}$-operator on the Dolbeault complex of $\Omega_I$ with the Neumann boundary condition. And they proved that the boundary map $\partial: K_0(\overline{\Omega}_I, \partial \Omega_I)\to K_1(\partial \Omega_I)$ maps $[D_N]$ to the fundamental class on $\partial \Omega_I$ naturally associated to the $CR$-structure on $\partial \Omega_I$. Furthermore, the fundamental class on $\partial \Omega_I$ is the spin$^c$ Dirac operator  on $\partial \Omega_I$ associated to the $CR$-structure. This identification of $\partial ([D_N])$ provides a different proof of Boutet de Monvel's Toeplitz index theorem \cite{boutet:index}.

In this article, we extend the above Baum-Douglas-Taylor result to the following case. Let $I$ be generated by $p_1, \cdots, p_M\in A=\mathbb{C}[z_1, \cdots, z_m]$ with $M\leq m-2$. We make the following assumptions. 
\begin{assumption}\label{ass:principal} 
\begin{enumerate}    
\item The Jacobian matrix $(\partial p_i/\partial z_j)_{i,j}$ is of maximal rank on the boundary $\partial \Omega_I=Z_I\cap \partial \mathbb{\ball}^m$; 
\item $Z_I$ intersects $\partial \mathbb{B}^m$ transversely.
\end{enumerate}
\end{assumption}
Under Assumption \ref{ass:principal}, $\Omega_I$ is an analytic space of complex dimension $k:=m-M\geq 2$ and complex codimension $M$. $\Omega_I$ has a smooth strongly pseudoconvex boundary $\partial \Omega_I=Z_I\cap \partial \ball^M$ and (possibly) a finite number of isolated singularities away from the boundary. Furthermore, by the assumption on the Jacobian matrix, $\Omega_I$ is a complete intersection space \cite[Sec.18.5]{ei:book}, from which we can conclude that the ideal $I\subset A$ is radical. 

Let $\Sigma_I$ denote the set of singular points of $\Omega_I$. The space $\Omega_I^0:=\Omega_I-\Sigma_I$ is a smooth submanifold of $\ball^m$ and inherits a natural riemannian metric $\sigma_I$ from the one on $\mathbb{B}^m$. Let $dV_I$ be the volume element on $\Omega_I^0$ defined by this metric $\sigma_I$. Consider the operator $D_N=\bar{\partial}+\bar{\partial}^*$ on the Dolbeault complex of $\Omega^0_I$ with the Neumann boundary condition. 

As $\partial \Omega_I$ is smooth and strongly pseudoconvex, the restriction of the complex structure to the boundary defines a $CR$-structure and therefore a spin$^c$ structure on $\partial \Omega_I$. The Dirac operator $D_{\partial \Omega_I}$ associated to this spin$^c$-structure is a  fundamental class of $K_1(\partial \Omega_I)$. The following theorem generalizes the results in \cite{bdt:boundary}. 
\begin{theorem}\label{thm:index}Under Assumption \ref{ass:principal}, the operator $D_N$ gives a relative $K$-homology class in $K_0(\overline{\Omega}_I, \partial \Omega_I)$. And the boundary map $\partial: K_0(\overline{\Omega}_I, \partial \Omega_I)$ maps $[D_N]$ to the fundamental class on $\partial \Omega_I$ which is associated to the $CR$-structure on $\partial \Omega_I$. 
\end{theorem}
We remark that Theorem  \ref{thm:index} holds true in much more general cases than Assumption \ref{ass:principal}. We refer the readers to Remark \ref{rmk:assump-1.6} for the more precise statement. 

Theorem \ref{thm:index} can be viewed as an analytic version of the Grothendieck-Riemann-Roch theorem. When $I$ is a homogeneous ideal of $A=\complex[z_1, \cdots, z_m]$, the zero variety $Z_I$ is invariant under the multiplication by $\complex^*=\complex-\{0\}$. Assume that the origin is the only possible isolated singular point of $Z_I$. $Z_I$ intersects with the unit sphere $\mathbb{S}^{2m-1}$ transversely, and therefore the boundary $\partial \Omega_I$ is a smooth submanifold of $\mathbb{S}^{2m-1}$. The group $S^1$ as the unit circle in $\complex^*$ acts on the unit sphere $\mathbb{S}^{2m-1}$ freely. As $I$ is a homogeneous ideal,  $\partial \Omega_I\subset \mathbb{S}^{2m-1}$ is invariant under the $S^1$-action. Furthermore the class $D_N$ (and $\partial [D_N]$) is $S^1$-equivariant and lives in $K^{S^1}_0(\overline{\Omega}_I, \partial \Omega_I)$ (and $K^{S^1}_1(\partial \Omega_I)$). The quotient space $X_I:=Z_I/S^1$ is an embedded smooth submanifold of $\complex \mathbb{P}^{m-1}$. Let $D_{X_I}$ be the fundamental class in $K_0(X_I)$ associated to the $\bar{\partial}$-operator on $X_I$. In Proposition \ref{prop:even-class}, we  explain that there is a natural isomorphism $\alpha_I$ from $K^{S^1}_1(\partial \Omega_I)$ to $K_0(X_I)$, mapping $\partial [D_N]$ to $[D_{X_I}]$.  
Let $\iota: \partial \Omega_I\hookrightarrow \partial \ball^m$ (and $i:X_I\hookrightarrow \mathbb{CP}^{m-1}$) denote the embedding map. A good understanding of $\iota_*(\partial [D_N])\in K^{S^1}_1(\mathbb{S}^{2m-1})$ will determine $i_*([D_{X_I}])\in K_0(\mathbb{CP}^{m-1})$ completely.
This provides an analytic approach to the Grothendieck-Riemann-Roch theorem for the projective variety $X_I$. In Theorem \ref{thm:index}, we do not assume that the ideal $I$ is homogeneous, and would like to view it as an analytic Grothendieck-Riemann-Roch Theorem.

Let $\rho$ be a defining function on $\Omega_I$, i.e. $\rho<0$ on $\Omega_I$, and $d\rho \ne 0$ on $\partial \Omega_I$. For $s\geq -1$, let $L^2_s(\Omega_I)$ be the ($s$-)weighted $L^2$-space on $\Omega_I$ with the norm defined by
\[
||f||_{s}^2=\int_{\Omega_I-\Sigma_I} |f|^2(-\rho)^s dV_I. 
\]
Let $L^2_{a, s}(\Omega_I)$ (and $L^p_{a,s}(\Omega_I)$) be the weighted Bergman space on $\Omega_I$ consisting of functions in $L^2_{s}(\Omega_I)$ (and $L^p_s(\Omega_I)$) that are  holomorphic on $\Omega_I^0$. $L^2_{a, s}(\Omega_I)$ is naturally a Hilbert $A$-module. Recall that a Hilbert $A$-module $(H, \alpha)$ (i.e. $\alpha: A\to \mathcal{L}(H)$) is essentially normal if $[\alpha(z_i), \alpha(z_j)^*]$ is compact for all $1\leq i, j\leq m$. We have the following corollary from Theorem \ref{thm:index}. 

\begin{corollary}
\label{cor:index} Under Assumption \ref{ass:principal}, for $s\geq 0$, 
\begin{enumerate}
\item The $A$-module $L^2_{a,s}(\Omega_I)$ is essentially normal. 
\item The $A$-module $L^2_{a, s}(\Omega_I)$ defines a $K$-homology class of the smooth boundary $\partial \Omega_I$, which is the fundamental class of $\partial \Omega_I$ defined by the $CR$-structure on $\partial \Omega_I$. 
\end{enumerate}
\end{corollary}

As an application, we use our index theorem, in particular Corollary \ref{cor:index}, to  study Hilbert modules associated to $I$. 
Let $L^2_a(\mathbb{B}^m)$ be the Bergman space on $\mathbb{B}^m$. The Toeplitz operators make $L^2_a(\ball^m)$ into an essentially normal Hilbert $A$-module.  Observe that the restriction map $R$ maps $f\in L^2_a(\ball^m)$ to a holomorphic function $f|_{\Omega_I}$ on $\Omega_I$.  Let $M$ be the complex codimension of $\Omega_I$ in $\ball^m$. When $\Omega_I$ is smooth and intersects with $\partial \ball^m$ transversely, Beatrous \cite{be:extension} proved that $R$ maps $L^2_a(\ball^m)$ continuously onto $L^2_{a, M}(\Omega_I)$. Using the developments in complex analysis \cite{ha:integral}, \cite{st:integral}, we have the following generalization of Beatrous' result. 

\begin{theorem} (Theorem \ref{thm:extension})\label{thm:restriction}  Under Assumption \ref{ass:principal}, there is a continuous linear operator $E:L^2_{a,M}(\Omega_I)\to L^2_a(\ball^m)$ such that $RE=Id$. Therefore, the restriction operator $R$ maps $L^2_a(\ball^m)$ continuously onto $L^2_{a,M}(\Omega_I)$. 
\end{theorem}

By Assumption \ref{ass:principal} and the dimension assumption $k\geq 2$, the Hartogs principle \cite[Ch.III, Ex. 3.5]{ha:book} holds on $\Omega_I$ and states that every holomorphic function on $\Omega_I^0$ is holomorphic on $\Omega_I$. Furthermore, Assumption \ref{ass:principal} plays a key role in the integral formula obtained in \cite{ha:integral}. Hence Assumption \ref{ass:principal} is crucial in Theorem \ref{thm:restriction} for $R$ to be surjective. In general, without Assumption \ref{ass:principal} the restriction map $R$ may fail to be surjective (c.f. Sec. \ref{subsec:ass}). We plan to study some cases in the near future when the range of $R$ has finite codimension.  

In Theorem \ref{thm:closure}, with Kai Wang we prove that, under Assumption \ref{ass:principal}, the kernel of the map $R$ is the closure $\bar{I}$ of $I$ in $L^2_a(\ball^m)$. Hence we have an exact sequence of Hilbert $A$-modules
\begin{equation}\label{eq:ex-sq}
0\longrightarrow \bar{I}\longrightarrow L^2_a(\ball^m)\longrightarrow L^2_{a, M}(\Omega_I)\longrightarrow 0.
\end{equation}
As the last two $A$-modules in the exact sequence (\ref{eq:ex-sq}) are essentially normal, we obtain the following theorem. We refer the reader to \cite{do-wa:quotient}, and \cite{gw:essential-normal2}--\cite{gz:quasi-homogeneous} for related results.

\begin{theorem}\label{thm:geometrization} (Theorem \ref{thm:essentially-normal}) Under Assumption \ref{ass:principal}, both $\bar{I}$ and the quotient $Q_I:=L^2_a(\ball^m)/\bar{I}$ are essentially normal $A$-modules. Furthermore, $Q_I$ and $L^2_{a,M}(\Omega_I)$ correspond to the same class in $K_1(\partial \Omega_I)$. 
\end{theorem}


Theorem \ref{thm:geometrization} confirms the conjecture by Arveson \cite{ar:conjecture} and the first author \cite{do:index} that the ideal $\bar{I}$ is  an essentially normal $A$-module when $I$ satisfies  Assumption \ref{ass:principal}. We refer to Remark \ref{rmk:schatten-p} for the discussion on the $p$-summability of the modules. This also suggests a good candidate for a fundamental class in $K_1(\partial \Omega_I)$ (and $K_0(X_I)$) when $\partial \Omega_I$ (and $X_I=\partial \Omega_I/S^1$) is not smooth. In algebraic geometry, when the zero variety $\Omega_I$ (and $X_I$) is not smooth, the Grothendieck-Riemann-Roch theorem usually involves resolutions of singularities by Hironaka's famous theorem \cite{hi:resolution}. Notice that resolutions \cite{hi:resolution} of a singularity variety are not unique. Therefore, it is hard to talk about a fundamental class on $X_I$ from the algebraic geometric point of view. We observe that the  quotient $A$-module $Q_I$ is always well defined without any requirements on $\Omega_I$. Theorem \ref{thm:geometrization} suggests that $Q_I=L^2_a(\ball^m)/\bar{I}$ may in general be a fundamental class in $K_1(\partial \Omega_I)$.

\begin{remark}\label{rmk:assumption} 
\begin{enumerate}
\item Assumption \ref{ass:principal}  can be weakened. For example, a natural case is that $I$ is generated by a finite number of holomorphic functions that are defined in a neighborhood of the closed ball $\overline{\ball}^m$ satisfying Assumption \ref{ass:principal} . All results in this article naturally extend to this case. We suggest the reader to compare our results to \cite{do-wa:quotient}, where $p$ is required to be a polynomial. 
\item Under Assumption \ref{ass:principal}, the ideal $I$ is radical. The extensions of our results to non radical ideals will be reported in the near future. 
\end{enumerate}
\end{remark}

This paper is organized as follows. In Sec. \ref{sec:resolution}, we will review some background knowledge about resolution of singularities, which is crucial in our proofs. In particular, we will explain the construction of the Hilbert space $L^2_{a,s}(\Omega_I)$. In Sec. \ref{sec:index}, we will present the construction of the $K$-homology class $D_N$, and prove Theorem \ref{thm:index}. In Sec. \ref{sec:restriction}, we will prove Theorem \ref{thm:restriction} about the restriction map $R$ and Theorem \ref{thm:geometrization} about the $A$-module structures on the ideal $\bar{I}$ and the quotient $Q_I$. We end this article with two remarks in Sec. \ref{sec:rmk}. In the Appendix  we prove together with Kai Wang that under Assumption \ref{ass:principal}, the closure $\overline{I}$ of $I$ in $L^2_a(\ball^m)$ agrees with the kernel of the operator $R$. \vskip 2mm

\noindent{\bf Acknowledgments}: In the middle of typing up this article, we received the preprint \cite{ee:geometric}. We would like to thank Miroslav Engli\v{s} and J\"org Eschmeier for sending us \cite{ee:geometric} and also discussions over emails. The results in \cite{ee:geometric} are deeply connected to results in this paper but also of independent interests with different approaches.  We refer to Remark \ref{rmk:ee} for more details. We would like to thank Paul Baum for discussions on the Grothendieck-Riemann-Roch theorem, and Jerome Kaminker for many discussions on index theory, and Kunyu Guo for bringing our attention to the issue about the closure of the ideal $I$ in $L^2_a(\ball^m)$. The second author would like to thank John McCarthy for pointing out Beatrous' work \cite{be:extension}, and suggesting the extension of Beatrous' work using the results in \cite{st:integral}, and Mohan Kumar for discussions about Assumption \ref{ass:principal} and the Grothendieck-Riemann-Roch theorem. All authors are partially supported by NSF. The second author is also partially supported by NSA. 
\section{Integration on the zero variety}\label{sec:resolution}

In this section, we provide a preliminary review of resolution of singularities and the construction of the Hilbert space $L^2_{a,s}(\Omega_I)$. 
\subsection{Resolution of singularities}
We recall some useful properties about resolution of $\Omega_I$. Hironaka \cite{hi:resolution} proved that every algebraic variety $V$ over $\mathbb{C}$ has a resolution. We will use the resolution method to study the zero variety $\Omega_I$.  More explicitly, there is a smooth manifold $\widetilde{\Omega}_I$ with a proper holomorphic surjection $\pi: \widetilde{\Omega}_I\to \Omega_I$ with the following properties:
\begin{enumerate}
\item  The exceptional set $E_I:=\pi^{-1}(\Sigma_I)$ is a hypersurface in $\Omega_I$ with (possible) ``normal crossing singularities" only.
\item  The restriction of $\pi: \widetilde{\Omega}_I-E_I\to \Omega_I-\Sigma_I$ is a biholomorphism. 
\end{enumerate}
The pullback $\pi^*\sigma_I$ is a positive semidefinite metric on $\widetilde{\Omega}_I$ degenerated on $E_I$. The pullback $\pi^*dV_I$ is a volume element on $\widetilde{\Omega}_I$ that vanishes on $E_I$. We choose a hermitian metric $\sigma$ on $\widetilde{\Omega}_I$. And denote $dV_\sigma$ to be the associated volume element on $\widetilde{\Omega}$. Define $d_{E_I}(x)$ to be the distance function on $\widetilde{\Omega}_I$ from $x$ to the exceptional subset $E_I$. In \cite[Eq. (9)]{fov:integral}, it is proved that there are positive constant $c, C, M$ such that 
\begin{equation}\label{eq:volume-estimate}
c d_{E_I}(x)^M dV_\sigma\leq \pi^*dV_I\leq C dV_\sigma,\ \ \text{on}\  \widetilde{\Omega}_I-E_I. 
\end{equation}

Without loss of generality, we may assume that $E_I$ is a divisor with only normal crossing, i.e. the irreducible components of $E$ are regular and meet complex transversely.  As is explained in \cite[Sec. 3]{or:compact-resolvent}, by Eq. (\ref{eq:volume-estimate}) there is an effective divisor $D$ of $\widetilde{\Omega}_I$ that is supported on $E_I$ such that  for $(p,q)$-forms $\Omega^{p,q}$ on $U$, 
\begin{equation}\label{eq:embedding}
L^2 (U,dV_\sigma, \Omega^{p,q}\otimes L_{-D})\subset L^2(U, \pi^*d V_I, \Omega^{p,q})\subset L^2(U, dV_\sigma, \Omega^{p,q}\otimes L_D),
\end{equation}
for $U$ a neighborhood of $E_I$ in $\widetilde{\Omega}_I$. 


\subsection{Weighted Bergman space} Let $L^2_{a,s}(\Omega_I)$ be the space of holomorphic functions on $\Omega_I-\Sigma$ that are square integrable with respect to the measure 
$(-\rho)^s dV_I$.  
\begin{lemma}\label{lem:bergman-space} $L^2_{a,s}(\Omega_I)$ is a closed subspace of $L^2_s(\Omega_I)$, and therefore a Hilbert space. 
\end{lemma}
\begin{proof}Let $L^2_{s}(\Omega_I)$ be the $s$-weighted $L^2$-space on $\Omega_I$ with respect to $(-\rho)^sdV_I$.  Pull back the space $L^2_{s}(\Omega_I)$  to the resolution $\widetilde{\Omega}_I$.  The pullback map is an isomorphism from $L^2_{s}(\Omega_I)$ to $L^2_s(\widetilde{\Omega}_I, \pi^*dV_I)$. By the inequality (\ref{eq:volume-estimate}) and (\ref{eq:embedding}), we have the following inclusion
\begin{equation}\label{eq:inclusion-L2}
L_s^2(\widetilde{\Omega}_I, dV_\sigma, L_{-D})\subset L_s^2(\widetilde{\Omega}_I, \pi^* dV_I)\subset L_s^2(\widetilde{\Omega}_I, dV_\sigma, L_D). 
\end{equation}

Consider the $\bar{\partial}$-equation $\bar{\partial}\varphi=0$ on $\widetilde{\Omega}_I$. Let $\ker \bar{\partial}|_{L_{-D}}$ (and $\ker\bar{\partial}|_{L_D}$) be the solution space of the $\bar{\partial}$-operator in $L_s^2(\widetilde{\Omega}_I, dV_\sigma, L_{-D})$ (and $L_s^2(\widetilde{\Omega}_I, dV_\sigma, L_D)$). Let $\ker\bar{\partial}|_{\Omega_I}$ be the  solution space of  the $\bar{\partial}|_{\Omega_I}$ operator in $L_s^2(\Omega_I)$ and therefore $L_s^2(\widetilde{\Omega}_I, \pi^* dV_I)$.  The following inclusion follows directly from  (\ref{eq:inclusion-L2}),
\[
\ker\bar{\partial}|_{L_{-D}}\subset \ker\bar{\partial}|_{\Omega_I}\subset \ker \bar{\partial}|_{L_D}.
\]

Next consider the the inclusion map
\[
i_*: \ker\bar{\partial}|_{L_{-D}}\subset \ker \bar{\partial}|_{L_D}. 
\]
Following the proof of \cite[Theorem 3.1]{or:compact-resolvent}, on a pseudoconvex neighborhood $W$ of a connected component of $E_I$, one has the following exact sequence
\[
0\longrightarrow \Gamma(W, L_{{-D}})\stackrel{i_*|_W}{\longrightarrow} \Gamma(W, L_D)\longrightarrow \Gamma(W, Q), 
\]
where $Q:=L_D/\big(i_*L_{-D}\big)$ is a coherent analytic sheaf with compact support in $W$.  As $Q$ is compactly supported, $\Gamma(W, Q)$ is of finite dimension. Therefore, $i_*|_W\big(\Gamma(W, L_{-D})\big)$ is of finite codimension in $\Gamma(W, L_D)$.  Globally, as $E_I$ only has finitely many components, $i_*:  \ker\bar{\partial}|_{L_{-D}}\subset \ker \bar{\partial}|_{L_D}$ is of finite codimension.  

As $ \ker\bar{\partial}|_{\Omega_I}$ is between $ \ker\bar{\partial}|_{L_{-D}}$ and $\ker \bar{\partial}|_{L_D}$, $L^2_{a,s}(\Omega_I)\cong L^2_{a,s}(\widetilde{\Omega}_I, \pi^* dV_I)$ is of finite codimension in $L^2_{a,s}(\widetilde{\Omega}_I, dV_\sigma, L_D)$. Therefore, $L^2_{a,s}(\Omega_I)$ is a closed subspace of $L^2_{a,s}(\widetilde{\Omega}_I, dV_\sigma, L_D)$ of finite codimension. As $L^2_{a,s}(\widetilde{\Omega}_I, dV_\sigma, L_D)$ is closed in $L^2_s(\widetilde{\Omega}_I, dV_\sigma, L_D)$, we conclude that $L^2_{a,s}(\Omega_I)$ is a closed subspace of $L^2_s(\Omega_I)=L^2_s(\widetilde{\Omega}, \pi^*dV_I)$ from Eq. (\ref{eq:inclusion-L2}), and therefore a Hilbert space. 
\end{proof}

\begin{remark}Similar arguments to the proof of Lemma \ref{lem:bergman-space} confirm that $L^p_{a,s}(\Omega_I)$ is a Banach space for $p\geq 1, s\geq -1$. 
\end{remark}

\section{An odd index theorem for analytic space with isolated singularity}\label{sec:index}
In this section, we explain the construction of the operator $D_{N}$ on $\Omega_I^0$ with the Neumann boundary condition and present the proof of Theorem \ref{thm:index}.  

In \cite{boutet:index}, Boutet de Monvel proved an index theorem for Toeplitz operators on a complex manifold with strongly pseudoconvex boundary.  Our result can be viewed as an extension of Boutet de Monvel's theorem to complex manifolds with isolated singularities. Such an extension was hinted by Boutet de Monvel \cite{boutet:index} and an approach was explained to the second author \cite{boutet:private}.  In the following development, we will take a different route by following the relative $K$-homology theory developed by Baum-Douglas-Taylor \cite{bdt:boundary}.  For simplicity, we will present our proofs below with the standard volume $dV_I$ on $\Omega_I$. The same results also hold true for the weighted volume element $(-\rho)^s dV_I$ ($s\geq 0$) with similar arguments. We also point out that although we have used the notation $\Omega_I$, the results in this section hold true for more general complex analytic spaces (See Remark \ref{rmk:assump-1.6}). 

\subsection{The $\bar{\partial}$-equation on $\Omega^0_I$} 

Let $\Omega^0_I$ denote $\Omega_I-\Sigma_I$. Let $\wedge^{p,q}T^*\Omega^0_I$ be the degree $(p,q)$ subbundle of $\wedge^{p+q}T^*\Omega^0_{I}$.  Recall that $\sigma_I$ is the subspace metric on $\Omega^0_I\subset \ball^m$. Denote $L^{p,q}(\Omega^0_I)$ to be the $L^2$-space on $\Omega^0_I$ associated to the metric $\sigma_I$. Consider the $\bar{\partial}_{p,q}$-operator on $L^{p,q}(\Omega^0_I)$
\[
\bar{\partial}_{p,q}: L^{p,q}(\Omega^0_I)\to L^{p, q+1}(\Omega^0_I), 
\] 
and its adjoint operator
\[
\bar{\partial}_{p,q}^*: L^{p,q+1}(\Omega^0_I)\to L^{p,q}(\Omega^0_I). 
\]

Let $r$ be a real valued function smooth in a neighborhood of $\partial \Omega_I$ satisfying $r=0$ on $\partial \Omega_I$ and $dr\ne 0$ on $\partial \Omega_I$. Define the $\bar{\partial}$-Neumann boundary condition by 
\[
\mathcal{D}^{p,q}:=\{ \xi \in \Gamma(\wedge^{p,q}T^*(\overline{\Omega}_I-\Sigma_I): (\bar{\partial} r)\lrcorner \xi=0\ \text{on}\ \partial \Omega_I\}.
\]
When $q=0$, the $\bar{\partial}$-Neumann boundary condition is trivial. Let $\bar{\partial}^N_{p,q}$ denote the closure of $\bar{\partial}_{p,q}$ restricted to $\mathcal{D}^{p,q}$. Consider the Laplace operator by 
\[
\Box^N_{p,q}=\bar{\partial}^N_{p,q-1}(\bar{\partial}^N_{p,q-1})^*+(\bar{\partial}^N_{p,q})^*\bar{\partial}^N_{p,q}: L^{p,q}(\Omega^0_I)\longrightarrow L^{p,q}(\Omega^0_I). 
\]
Since $\partial \Omega_I$ is strongly pseudoconvex, the Laplace operator $\Box^N_{0,q}$ ($q\geq 1$) with the $\bar{\partial}$-Neumann boundary condition has compact  resolvent on the resolution $\widetilde{\Omega}_I$. Hence, \cite[Theorem 1.1]{or:compact-resolvent} implies that for $q\geq 1$, 
\[
\Box^N_{0,q}: L^{0,q}(\Omega^0_I)\to L^{0,q}(\Omega^0_I)
\]
has compact resolvent.  

We remark that when the dimension of $\Omega_I$ is greater than or equal to 2, under Assumption \ref{ass:principal} the Hartogs principle \cite[Ch.III, Ex.3.5]{ha:book} implies that every function in $\ker \bar{\partial}^N_{0,0}=L^2_a(\Omega_I)$ is holomorphic on the whole $\Omega_I$.

\subsection{A Hilbert module}

In this section, following \cite{bdt:boundary}, we construct a $K$-homology class $D_N$ in $K_0(\overline{\Omega}_I, \partial \Omega_I)$.  

Let $H_0(\Omega_I^0)=\oplus_{q\geq 0} L^{0,2q}(\Omega^0_I)$ (and $H_1(\Omega_I^0)=\oplus_{q\geq 0} L^{0, 2q+1}(\Omega^0_I)$) be the Hilbert space of $(0,even)$ (and $(0,odd)$) forms on $\Omega^0_I$.  We consider the differential operator 
\[
D_N:= \bar{\partial}^N+(\bar{\partial}^N)^*: H_0(\Omega_I^0)\longrightarrow H_1(\Omega^0_I).
\]
The operator $D_N$ is a first order differential operator on $\Omega_I$. Denote $\sigma_{D_N}(x, \xi)$ to be the symbol of $D_N$. Let $\mathcal{D}(D_N)$ denote the domain of $D_N$ in $H_0(\Omega_I^0)$. 

Let $C(\overline{\Omega}_I)$ be the $C^*$-algebra of continuous functions on the closure $\overline{\Omega}_I$. Consider the multiplication of $C(\overline{\Omega}_I)$ on $H_0$ and $H_1$. Denote the corresponding $*$-representations by $\sigma_0: C(\overline{\Omega}_I)\to \mathcal{L}(H_0)$ and $\sigma_1: C(\overline{\Omega}_I)\to \mathcal{L}(H_1)$. 

\begin{proposition}\label{prop:kk-class}The graded Hilbert space $H:=H_0\oplus H_1$, and the representation 
\[
\sigma:=\sigma_0\oplus \sigma_1: C(\overline{\Omega}_I)\longrightarrow \mathcal{L}(H), 
\]
and the operator 
\[
A=\left(\begin{array}{cc}0&D_N^*\\ D_N&0\\ \end{array}\right)
\]
form an unbounded Kasparov module for $C(\overline{\Omega}_I)$ and its ideal $C_0(\overline{\Omega}_I)$ which consists of functions in $C(\overline{\Omega}_I)$ vanishing at the boundary $\partial \Omega_I$. 
\end{proposition}

\begin{proof}

Let $C^\infty(\overline{\Omega}_I)$ be the space of continuous functions on $\overline{\Omega}_I$ whose pullback to the resolution $\widetilde{\Omega}_I$ via the map $\pi$ is smooth on the closure $\overline{\widetilde{\Omega}}_I$. Since the singularities on $\Omega_I$ are isolated, by a partition of unity, we can easily show that $C^\infty(\overline{\Omega}_I)$ is a dense $*$-subalgebra of $C(\overline{\Omega}_I)$. One quickly checks that for any $u\in H^1_{loc}(\Omega_I^0, \wedge^{0, even} T^*\Omega^0_I )$
\[
D_N\sigma_0(f)u-\sigma_1(f)D_Nu=\sigma_{D_N}(x, df)u,\ \forall f\in C^\infty(\overline{\Omega}_I). 
\]
As the metric on $\Omega_I^0$ vanishes toward singular points, $\sigma_{D_N}(x, df)$ extends to a bounded linear operator from $H_0$ to $H_1$. From \cite[Prop. 1.3]{bdt:boundary}, we can conclude that for every $f\in C^\infty(\overline{\Omega}_I)$, $\sigma_1(f)$ preserves the domain of $D_N^*: H_1(\Omega^0_I)\to H_0(\Omega^0_I)$, and $[\sigma(f), D_N]$ extends to a bounded operator on $H_0(\Omega_I^0)\oplus H_1(\Omega_I^0)$.  And, therefore $[\sigma(f), D_N]$ is bounded for all $f\in C(\overline{\Omega}_I)$. 

As the resolution $\widetilde{\Omega}_I$ is a complex manifold with strongly pseudoconvex boundary,  Kohn \cite{ko:compact} showed that the $\bar{\partial}$-laplacian $\Box^N_{0,q}$ has a finite dimensional kernel and  a compact solution operator on $L^{0,q}(\widetilde{\Omega}_I)$ for all $q\geq 1$. By \cite[Thm. 1.2]{or:compact-resolvent}, the $\bar{\partial}$-laplacian $\Box^N_{0,q}$ also has a finite dimensional kernel and a compact solution operator on $L^{0,q}(\Omega^0_I)$, for $q\geq 1$.  From this, we can derive that $\Box^N_{0, q}$ on $L^{0,q}(\Omega^0_I)$ has a compact resolvent for $q\geq 1$. Therefore,  the operator 
\[
D_ND_N^*: H_1(\Omega^0_I)\to H_0(\Omega^0_I)
\]
has a compact resolvent. 

For $f\in C_0^\infty(\overline{\Omega}_I)$, the operator $\sigma_0(f)(D_N^*D_N+1)^{-1}$ maps $L^2(\Omega^0_I, \wedge^{0, even}T^*\Omega^0_I)$ to the Sobolev space $H^{2,2}_0(\Omega_I^0, \wedge^0T^*\Omega^0_I)$, where $H^{-, 2}_0$ is the $L^2$-Sobolev space of sections that vanishes on $\partial \Omega_I$. Hence the operator 
$\sigma_0(f)(D_N^*D_N+1)^{-1}$
is compact on $L^2(\Omega_I^0, \wedge^{0, even}T^*\Omega^0_I)$ by the Rellich\footnote{Recall that the volume element on $\Omega^0_I$ vanishes at the singular points of a certain order, i.e. Eq. (\ref{eq:volume-estimate}). Using this fact, one can prove a Rellich compact embedding theorem for $H^{2,2}_0(-)\hookrightarrow L^2(-)$ in the same way as  \cite{ev:book}.} compact embedding theorem. Analogously, the operator $\sigma_1(f)(D_ND_N^*+1)^{-1}$ is compact on $L^2(\Omega_I^0, \wedge^{0, odd}T^*\Omega^0_I)$. 

By \cite[Prop. 1.1, 1.39, Prop. 3.1]{bdt:boundary}, we conclude $(H, \sigma, A)$ is an unbounded Kasparov module.  
\end{proof}

Besides the $\bar{\partial}$-Neumann boundary condition, we can consider other boundary conditions. For example, denote $D_{max}$ and $D_{min}$ to be the maximal and minimal extension of the first order differential operator $D=\bar{\partial}+\bar{\partial}^*$ on $C^\infty_c(\Omega^0_I, \wedge^{0,even} T^*\Omega^0_I)$. More explicitly, the domain $\mathcal{D}(D_{max})$ is 
\[
\{\xi\in L^{2}(\Omega^0_I, \wedge^{0, even}T^* \Omega^0_I)\ :\ D\xi \in L^2(\Omega^0_I, \wedge^{0, odd}T^*\Omega^0_I)\}. 
\]
Let $D^t: C_c^\infty(\Omega^0_I)\to C_c^\infty(\Omega^0_I)$ be the formal adjoint of $D$. We have $D^t_{max}=(D_{min})^*$, and 
\[
D_{min}\subset D_N\subset D_{max}. 
\] 

\begin{proposition}\label{prop:unique-ext}The operator $D_{max}$ (and $D_{min}$) on $(H, \sigma)$ defines a $K$-cycle $[D_{max}]$ (and $[D_{min}]$) for $K_0(\overline{\Omega}_I, \partial\Omega_I )$. Furthermore,  in $K_0(\overline{\Omega}_I, \partial \Omega_I)$
\[
[D_{max}]=[D_{min}]=[D_N]. 
\] 
\end{proposition}
The proof of Proposition \ref{prop:unique-ext} is a straightforward extension of \cite[Prop. 2.1, 3.1, 3.3]{bdt:boundary}. We skip the detail to avoid repetition. 

\subsection{The boundary map in K-homology}
In \cite{bd:theory}, Baum and the first author developed a long exact sequence for relative $K$-homology. In particular, applying the long exact sequence to our study, we obtain a boundary map $\partial: K_0(\overline{\Omega}_I, \partial \Omega_I)\longrightarrow K_1(\partial \Omega_I)$.  In this subsection, we study the boundary $\partial [D_N]\in K_1(\partial \Omega_I)$.  

Let $\ker(D_N)$ be the space 
\[
\{\xi\in L^2(\Omega^0_I, \wedge^{0,even}T^*\Omega^0_I)\ :\ D_N (\xi)=0\}.
\]
By the property that $D_N$ has a finite dimensional solution space on $L^2(\Omega^0_I, \wedge^{0,q}T^*\Omega^0_I)$ for $q\geq 1$, we know that up to a finite dimensional subspace $\ker(D_N)$ is equal to $L^2_a(\Omega_I)$, the space of $L^2$-holomorphic functions on $\Omega^0_I$. The K-homology class in $K_1(\partial \Omega_I)$ associated to $\ker(D_N)$ is equal to the K-homology class associated to $L^2_a(\Omega_I)$, i.e. 
\[
[\ker(D_N)]=[L^2_a(\Omega_I)]. 
\]

As $\partial \Omega_I$ is a strongly pseudoconvex, the restriction of the complex structure on $\Omega_I$ to the boundary defines a $CR$-structure on $\partial \Omega_I$, and therefore a spin$^c$ structure on $\partial \Omega_I$. Let $D_{\partial \Omega_I}$ be the Dirac operator associated to this $CR$-structure. Then we can conclude from \cite[Prop. 4.5, 4.6]{bdt:boundary} that in $K_1(\partial \Omega_I)$
\begin{equation}\label{eq:k-homology}
\partial [D_N]=[\ker D_N]=[L^2_a(\Omega_I)]=[D_{\partial \Omega_I}]. 
\end{equation}

\subsection{Proof Theorem \ref{thm:index} and Corollary \ref{cor:index}}

Theorem \ref{thm:index} is a corollary of Proposition \ref{prop:kk-class} and Eq. (\ref{eq:k-homology}). We explain the proof of Corollary \ref{cor:index}.  By Proposition \ref{prop:kk-class}, \ref{prop:unique-ext}, and Eq. (\ref{eq:k-homology}), we conclude that the $L^2_{a}(\Omega_I)$ defines a $K$-homology class on $\partial \Omega_I$. This implies that $L^2_a(\Omega_I)$ is an essentially normal $A$-module, and confirms Part (i) of Theorem \ref{thm:index}. Also we conclude from Eq. (\ref{eq:k-homology}) that $[L^2_a(\Omega_I)]$ is equal to the fundamental class of $\partial \Omega_I$ associated to the canonical $CR$-structure and therefore the contact structure on $\partial \Omega_I$, and confirms Part (ii) of Theorem \ref{thm:index}. 

\begin{remark}\label{rmk:assump-1.6}We point out that the proofs of Theorem \ref{thm:index} only use the property that $\Omega_I$ is a complex analytic space of pure dimension $n$  with the following properties.
\begin{enumerate}
\item $\Omega_I$ has a strongly pseudoconvex boundary $\partial \Omega_I:=\overline{\Omega}_I\backslash \Omega_I$.
\item $\Omega_I$ may contain isolated singularities away from $\partial \Omega_I$. 
\end{enumerate}
From these assumptions, we can conclude from \cite[Theorem 1.2]{or:compact-resolvent} that that operator $D_ND_N^*$ has compact resolvents and therefore Proposition \ref{prop:kk-class}, \ref{prop:unique-ext}, and Theorem \ref{thm:index}. 
\end{remark}

When $I$ is a homogeneous ideal of $A=\complex[z_1, \cdots, z_m]$, the zero variety $Z_I$ is invariant under the multiplication by $\complex^*=\complex-\{0\}$. Assume that the origin is the only possible isolated singular point of $Z_I$. $Z_I$ intersects with the unit sphere $\mathbb{S}^{2m-1}$ transversely, and therefore the boundary $\partial \Omega_I$ is a smooth submanifold of $\mathbb{S}^{2m-1}$. The group $S^1$ as the unit circle in $\complex^*$ acts on the unit sphere $\mathbb{S}^{2m-1}$ freely. As $I$ is a homogeneous ideal,  $\partial \Omega_I$ is invariant under the $S^1$-action. We easily observe that the $K$-homology class $[D_N]$ (and $\partial [D_N]$) is $S^1$-equivariant and lives in $K^{S^1}_0(\overline{\Omega}_I, \partial \Omega_I)$ (and in $K^{S^1}_1(\partial \Omega_I)$). Furthermore, the quotient space $X_I:=\partial \Omega_I/S^1$ is an embedded smooth submanifold of $\complex \mathbb{P}^{m-1}$, which is the projective variety associated to the ideal $I$. Let $D_{X_I}\in K_0(X_I)$ be the fundamental class on $X_I$ associated to the $\bar{\partial}$-operator.  We explain below the relation between $\partial [D_N]\in K^{S_1}_1(\partial \Omega_I)$ and $[D_{X_I}]\in K_1(X_I)$.  

\begin{proposition}\label{prop:even-class}When the ideal $I$ is homogeneous and the origin is the only possible singular point of $Z_I$, there is a natural isomorphism $\alpha_{I}$ from $K_1^{S^1}(\partial \Omega_I)$ to $K_0(X_I)$ such that $\alpha_I\big(\partial [D_N]\big)=[D_{X_I}]$ in $K_0(X_I)$. 
\end{proposition}

\begin{proof}
We observe that as $I$ is homogeneous, the $CR$-structure on $\partial \Omega_I$ is $S^1$-equivariant, and gives an $S^1$-equivariant spin$^c$ structure on $\partial \Omega_I$. As $\dim(\partial \Omega_I)$ is odd, the $S^1$-equivariant Poincar\'e duality gives an isomorphism 
\[
\big(PD^{S^1}_{\partial \Omega_I}\big)^{-1}: K_1^{S^1}(\partial \Omega_I)\stackrel{\cong}{\longrightarrow} K^0_{S^1}(\partial \Omega_I). 
\]
As the $S^1$-action on $\partial \Omega_I$ is free, $K^0_{S^1}(\partial \Omega_I)$ is naturally isomorphic to $K^0(X_I)$. Let $\beta_I$ denote the isomorphism from $K^0_{S^1}(\partial \Omega_I)$ to $K^0(X_I)$. $X_I$ is a complex manifold with a canonical spin$^c$ structure. The Poincar\'e duality gives an isomorphism 
\[
PD_{X_I}: K^0(X_I)\longrightarrow K_0(X_I). 
\]
Define $\alpha_I: K_1^{S^1}(\partial \Omega_I)\to K_0(X_I)$ to be the composition $PD_{X_I}\circ \beta_I \circ \big(PD^{S^1}_{\partial \Omega_I}\big)^{-1}$.  $\alpha_I$ is obviously an isomorphism as each involved component is.  

Theorem \ref{thm:index} identifies the class $\partial [D_N]$ with the fundamental class $D_{\partial \Omega_I}$ associated to the $CR$-structure on $\partial \Omega_I$. Furthermore, it is not hard to trace through the arguments that this is an identification in $K^{S^1}_1(\partial \Omega_I)$. The inverse Poincar\'e duality map $\big(PD^{S^1}_{\partial \Omega_I}\big)^{-1}$ maps $D_{\partial \Omega_I}$ to the trivial line bundle $\complex_{\partial \Omega_I}$ in $K^0_{S^1}(\partial \Omega_I)$. The map $\beta_I$ maps $\complex_{\partial \Omega_I}$ to the trivial line bundle $\complex_{X_I}$ in $K^0(X_I)$. And the Poincar\'e duality map $PD_{X_I}$ on $X_I$ maps $\complex_{X_I}$ to the fundamental class $D_{X_I}$ in $K_0(X_I)$. Therefore, we conclude that $\alpha_I$ maps $\partial [D_N]$ to $[D_{X_I}]$ in $K_0(X_I)$. 
\end{proof}

\section{Geometrization of the quotient Hilbert module}\label{sec:restriction}
In this section, we construct a right inverse $E$ of the restriction operator $R: L^2_{a}(\ball^m)\to L^2_{a,M}(\Omega_I)$, and apply it to prove that both the closure $\overline{I}$ of $I$ in $L^2_a(\ball^m)$  and the quotient $Q_I=L^2_a(\ball^m)/\overline{I}$ are essentially normal $A$-modules. We prove a Toeplitz index theorem for $Q_I$ by identifying the $K$-homology class associated to $Q_I$ with the fundamental class on $\partial \Omega_I:=\overline{\Omega}_I\backslash \Omega_I$. 
\subsection{Integral representation formula}
In \cite[Theorem I1]{ha:integral}, an integral representation formula was obtained for an analytic space satisfying Assumption \ref{ass:principal}. More precisely,  let $\alpha$ be the volume form on $\partial \Omega_I$. There is a kernel function $K(z, \zeta)$ on $\Omega_I \times \partial \Omega_I$ such that for each $\zeta\in \partial \Omega_I$, the function $K(-, \zeta)$ is holomorphic on $\overline{\Omega}_I$.  Let $f$ be a holomorphic function on $\overline{\Omega}_I$. Then $f$ can be represented by the following integral 
\begin{equation}\label{eq:integration}
f(z)=\int_{\partial \Omega_I} f(\zeta) \frac{K(z, \zeta)}{(1-\sum_{i=1}^m\bar{\zeta}_iz_i)^k}\alpha(\zeta) . 
\end{equation}

By choosing a cut-off function $\chi$ supported in a neighborhood of the boundary $\partial \Omega_I$, we can conclude the following corollary from Equation (\ref{eq:integration}). 
\begin{lemma}\label{lem:integration}
Under Assumption \ref{ass:principal}, there is neighborhood $M_I$ of $\Omega_I$ in $Z_I$, and a smooth differential form $\eta_j(z, \zeta)$ ($j=0,1,\cdots$) on  $M_I\times M_I $ of bidegree $(k,k)$ (when $j=0$, $(k,k-1)$) in $\zeta$ supported away from $\Sigma_I$, the set of singular points, and $(0,0)$ in $z$ such that $\eta_j(-, \zeta)$ is holomorphic on $M_I$ for any $\zeta\in M_I$. For any  $f\in L^1_{a,s-1}(\Omega_I)$, the following integral representations hold,
\[
f(z)=\int_{\partial \Omega_I} \frac{f(\zeta) \eta_0(z, \zeta)}{ (1-\sum_{i=1}^m z_i\bar{\zeta}_i)^{k}},\ s=0,
\]
and
\[
f(z)=\int_{\Omega_I} \frac{f(\zeta) \eta_s(z, \zeta)(-\rho)^{s-1}(\zeta)}{ (1-\sum_{i=1}^m z_i\bar{\zeta}_i)^{k+s}},\ s\geq 1.
\]
 \end{lemma}

\begin{proof} The proof is a line by line repetition of the proof of \cite[Corollary 2.4]{be:extension} from the representation (\ref{eq:integration}). It is worth pointing out that the cut-off function $\chi$ in the proof of \cite[Corollary 2.4]{be:extension} can be chosen to be supported away from the set $\Sigma_I$ of singular points in $\Omega_I$. Therefore, the support of $\eta_j(-, \zeta)$ is also away from $\Sigma_I$. We remark that the property that the dimension $k$ of $\Omega_I$ is at least 2 and Assumption \ref{ass:principal} assures that the Hartogs principle \cite[Ch.III, Ex.3.5]{ha:book} holds on $\Omega_I$.  The Hartogs principle assures that any $f\in L^1_{a,s}(\Omega_I)$ is holomorphic on $\Omega_I$.  Therefore, the singularities in $\Omega_I$ do not affect any parts of the proofs in \cite{be:extension}. We leave the detail to the reader. 
\end{proof}
\subsection{Extension operator} Let $C(\mathbb{B}^m)$ (and $C(\Omega_I)$) be the space of continuous functions on $\mathbb{B}^m$ (and $\Omega_I$). Consider the restriction operator $R:C(\mathbb{B}^m)\to C(\Omega_I)$ defined by 
\[
R(f)(z):=f(z),\ \text{for}\ z\in \Omega_I. 
\]
The restriction of $R$ to the subspace $\mathcal{O}(\mathbb{B}^m)$ of holomorphic functions on $\ball^m$ takes image in $\mathcal {O}(\Omega_I)$.  

The function $\eta_0(-, \zeta)$ can be viewed as a holomorphic function from $\Omega_I$ to smooth $(k,k-1)$ forms on $M_I$. By \cite[Corollary 12.1]{bungart:extension}, $\eta_0(-,\zeta)$ can be extended to a holomorphic function on $\ball^m$ with value in smooth $(k,k-1)$ forms on $M_I$ supported away from the set $\Sigma_I$ of singular points, which will be again denoted by $\eta_0(z, \zeta)$. Define a linear operator $E_0: L^{1}_{a, -1}(\Omega_I)\rightarrow \mathcal{O}(\ball^m)$ by 
\[
E_0(f)(z)=\int_{\partial \Omega_I}\frac{f^*(\zeta) \eta_0(z, \zeta)}{ (1-\sum_{i=1}^m z_i\bar{\zeta}_i)^{k}},
\]
where $f^*$ is the boundary value of $f$. 

Similarly we extend $\eta_j(z,\zeta)$ to a holomorphic function on $\ball^m$ with value in smooth $(k,k)$ forms on $M_I$ (a neighborhood of $\Omega_I$ in $Z_I$) supported away from $\Sigma_I$. Define a linear operator $E_j: L^1_{a,j}(\Omega_I)\rightarrow \mathcal{O}(\ball^m)$ by 
\[
E_jf(z)=\int_{\Omega_I} \frac{f(\zeta) \eta_j(z, \zeta)(-\rho)^{j-1}(\zeta)}{ (1-\sum_{i=1}^m z_i\bar{\zeta}_i)^{k+j}},\ j\geq 1. 
\]

A direction generalization of \cite[Theorem 2.5, Corollary 2.6]{be:extension} proves the following proposition. 
\begin{proposition}\label{prop:extension}
\begin{enumerate}
\item For every $f\in L^1_{a, -1}(\Omega_I)$, $E_0(f)\in \mathcal{O}(\overline{\ball}^m\backslash \partial \Omega_I)$ and $E_0(f)|_{\Omega_I}=f$.
\item For $j\geq 1$, every $f\in L^1_{a, j-1}(\Omega_I)$, $E_j(f)\in \mathcal{O}(\ball^m)$ and 
\[
E_jf|_{\Omega_I}=f,\qquad E_jf=E_kf,\ k<j. 
\]
\end{enumerate}
\end{proposition}

The following theorem is a strengthened version of Theorem \ref{thm:restriction}.
\begin{theorem}\label{thm:extension} 

Under Assumption \ref{ass:principal}, there is a continuous linear operator 
\[
E: L^p_{a, s+M}(\Omega_I)\longrightarrow L^p_{a,s}(\ball^m),\ 1<p<\infty, -1\leq s,
\]
such that $RE=Id$. Therefore, the restriction operator $R$ is a surjective bounded linear operator 
\[
R: L^p_{a,s}(\ball^m)\longrightarrow L^p_{a, s+M}(\Omega_I), \ 1<p<\infty, -1\leq s.
\]
\end{theorem}
\begin{proof} The proof is a direct generalization of the proof \cite[Thm. 1.1]{be:extension}.  The key observation is that with the choice of the cut-off function $\chi$ in the construction of Lemma \ref{lem:integration}, $\eta_j(-, \zeta)$ ($j\geq 0$) is zero outside a neighborhood of the boundary $\partial \Omega_I$.  Therefore, the support of $\eta_j(z, \zeta)$ is away from singular $\zeta$ values in $\Omega_I$. Hence, the kernel function $\frac{\eta_j(z, \zeta)\rho^{j-1}(\zeta)}{ (1-\sum_{i=1}^m z_i\bar{\zeta}_i)^{k+j}}$ is bounded on the whole $\Omega_I$ by 
\[
\frac{CdV_I}{(1-\sum_{i=1}^m z_i\bar{\eta}_i)^{k+j}} 
\]
for some constant $C>0$. This property allows us to use \cite[Theorem 4.1]{be:extension} to conclude the desired statements on the bounds of the operator $R$. 
\end{proof}
\subsection{Equivalence of Hilbert modules}

We look at the restriction operator 
\[
R: L^2_{a}(\ball^m)\longrightarrow L^2_{a, M}(\Omega_I).
\] 
As $R$ maps all functions in $\overline{I}$ to the zero function on $\Omega_I$, $R$ restricts to a  bounded linear map
\[
R_{\Omega_I}: Q_I=L^2_a(\ball^m)/\bar{I}\longrightarrow L^2_{a, M}(\Omega_I). 
\]
The extension map $E$ with $RE=I$ implies that $R_{\Omega_I}$ is surjective. 

To study the above structure, we prove the following general fact. 

\begin{proposition}\label{prop:equivalence}
Let $H_1$ and $H_2$ be two Hilbert $A$-modules. Let $X:H_1\to H_2$ be an isomorphism of $A$-modules. Any two of the following three statements imply the third one. 
\begin{enumerate}
\item The $A$-module $H_1$ is essentially normal.
\item The $A$-module $H_2$ is essentially normal.
\item The operator $X^*X: H_1\rightarrow H_1$ commutes with the $A$-module structure up to compact operators.  
\end{enumerate}
Therefore, when $H_1$ and $H_2$ are isomorphic essentially normal $A$-modules, the corresponding extensions associated to $H_1$ and $H_2$ are unitarily equivalent. 
\end{proposition}

\begin{proof}

By polar decomposition, we write 
\[
X=US, 
\]
where $S: H_1\rightarrow H_1$ is an invertible positive operator, and $U: H_1\rightarrow H_2$ is a unitary operator.  

We observe that $X$ satisfies that for any $p\in A$ and $\xi \in H_1$, 
\[
\sigma_2(p)X(\xi) =X \sigma_1(p) (\xi), 
\]
where $\sigma_i(p)$ is the representation of $p$ on $H_i$. By the polar decomposition of $X$,  we can write 
\begin{equation}\label{eq:sigma-2}
\sigma_2(p)=X\sigma_1(p)X^{-1}=US\sigma_{1}(p)S^{-1}U^*. 
\end{equation}

We compute $X^*X$ by
\[
X^*X=SU^*US=S^2. 
\]
If $X^*X=S^2$ commutes with $\sigma_1$ up to compact operators,  $S$ commutes with $\sigma_1$ up to compact operators too.  Equation (\ref{eq:sigma-2}) implies that in the Calkin algebra $\calc(H_1)$,
\begin{equation}\label{eq:equal-calkin}
\overline{\sigma_2(p)}=\overline{US\sigma_1(p)S^{-1}U^*}=(\overline{U}) (\overline{\sigma_1(p)}) (\overline{U})^*,
\end{equation}
where we have used $\overline{T}$ to denote the image of an operator $T$ in the corresponding Calkin algebra. 

The Equation (\ref{eq:equal-calkin}) shows that the unitary operator $U$ is a unitary equivalence between $\overline{\sigma}_1$ and $\overline{\sigma}_2$. From this we can conclude that 
\[
(1)\ \&\ (3)\Longrightarrow (2),\ \text{and}\ (2)\ \&\ (3)\Longrightarrow (1). 
\]

In the following we show that 
\[
(1)\ \&\ (2)\Longrightarrow (3). 
\]

By (1) and (2), we can extend $\overline{\sigma}_1:A\to \calc(H_1)$ and $\overline{\sigma}_{2}:A\to \calc(H_2)$ to  $*$-algebra morphisms 
\[
\overline{\sigma}_1: C(\overline{\ball}^m)\to \calc(H_1),\ \text{and}\ \overline{\sigma}_2: C(\overline{\ball}^m)\to \calc(H_2). 
\]
By the Fuglede-Putnam theorem and the assumption that both $\sigma_{1}(p)$ and $\sigma_{2}(p)$ are essentially normal for $p\in A$, Equation (\ref{eq:sigma-2}) implies that 
\[
\overline{\sigma}_{2}(p)^*=\overline{X}\overline{\sigma}_{1}(p)^*\overline{X}^{-1},
\]
and therefore
\[
\overline{\sigma}_2(p^*)=\overline{X}\overline{\sigma}_{1}(p^*)\overline{X}^{-1}.
\]
We conclude that for any $a\in C(\overline{\ball}^m)$, 
\begin{equation}\label{eq:similarity}
\overline{\sigma}_2(a)=\overline{X}\overline{\sigma}_{1}(a)\overline{X}^{-1}=(\overline{U})(\overline{S})\bar{\sigma}_{1}(a)(\overline{S})^{-1}(\overline{U})^*. 
\end{equation}
Taking the adjoint of the both sides of Eq. (\ref{eq:similarity}), we have 
\begin{equation}\label{eq:adjoint}
\overline{\sigma}_{2}(a^*)=\overline{\sigma}_{2}(a)^*=(\overline{X}^{-1})^*\overline{\sigma}_{1}(a)^*(\overline{X})^*=(\overline{U})(\overline{S})^{-1}\overline{\sigma}_{1}(a^*)(\overline{S})(\overline{U})^*,\ \text{for all}\ a\in C(\overline{\ball}^m). 
\end{equation}

Comparing Eq. (\ref{eq:similarity}) and Eq. (\ref{eq:adjoint}), we obtain the following equality, for all $a\in C(\overline{\ball}^m)$
\[
(\overline{S})\overline{\sigma}_{1}(a)(\overline{S})^{-1}=(\overline{S})^{-1}\overline{\sigma}_{1}(a)(\overline{S}),
\]
and 
\[
(\overline{S})^2\overline{\sigma}_{1}(a)=\overline{\sigma}_{1}(a)(\overline{S})^2.
\]
We conclude from the last equality that $\sigma_1(a)$ commutes with $S^2=X^*X$ up to compact operators. And therefore $\overline{\sigma}_1$ and $\overline{\sigma}_2$ are unitarily equivalent by (\ref{eq:equal-calkin}).
\end{proof}

Notice that both $Q_I$ and $L^2_{a, M}(\Omega_I)$ are $A$-modules. Denote $\sigma_{Q_I}$ and $\sigma_{\Omega_I}$ to be the corresponding morphisms from $C(\overline{\ball}^m)$ to $\mathcal{C}(Q_I)$ and $\mathcal{C}(\Omega_I)$, where $\calc(Q_I)$ and $\calc(\Omega_I)$ are the Calkin algebras on $Q_I$ and $L^2_{a,M}(\Omega_I)$.  Next theorem studies the relation between the two extension classes $\sigma_{Q_I}$ and $\sigma_{\Omega_I}$. 
\begin{theorem}\label{thm:essentially-normal}  Under Assumption \ref{ass:principal}, both $\bar{I}$ and the quotient $Q_I:=L^2_a(\ball^m)/\bar{I}$ are essentially normal $A$-modules. Furthermore, $Q_I$ and $L^2_{a,M}(\Omega_I)$ correspond to the same class in $K_1(\partial \Omega_I)$. 
\end{theorem}

\begin{proof} We consider the exact sequence of $A$-modules, 
\[
0\longrightarrow \ker{R_{\Omega_I}}\stackrel{\iota}{\longrightarrow} L^2_a(\ball^m)\stackrel{R}{\longrightarrow} L^2_{a,M}(\Omega_I)\longrightarrow 0. 
\]
By Theorem \ref{thm:extension},  $R$ is a surjective $A$-module map. According to Corollary \ref{cor:index}, $L^2_{a,M}(\Omega_I)$ is an essentially normal $A$-module. We conclude from \cite[Thm. 1]{do:index} that the kernel $\ker R$ is an essentially normal $A$-module as $L^2_a(\ball^m)$ is also essentially normal. By Theorem \ref{thm:closure}, the closure $\bar{I}$ of $I$ in $L^2_a(\ball^m)$ agrees with the kernel $\ker R $.  Hence, $\bar{I}$ is an essentially normal $A$-module. 

Now consider the following exact sequence
\[
0\longrightarrow \bar{I}\longrightarrow L^2_a(\ball^m)\longrightarrow Q_I\longrightarrow 0. 
\] 
By \cite[Thm. 1]{do:module}, the essentially normality of  the ideal $\bar{I}$ implies that the quotient module $Q_I=L^2_a(\ball^m)/\bar{I}$ is essentially normal. 

As $\overline{I}=\ker{R}$, $R_{\Omega_I}: Q_I\to L^2_{a,M}(\Omega_I)$ is an isomorphism of essentially normal $A$-modules. Apply Proposition \ref{prop:equivalence} to $R_{\Omega_I}$. We conclude that $Q_I$ and $L^2_{a,M}(\Omega_I)$ give rise to the same $K$-homology class in $K_1(\partial \Omega_I)$. 
\end{proof}

\begin{remark} \label{rmk:schatten-p}Theorem \ref{thm:essentially-normal} states that $Q_I$ and $L^2_{a, M}(\Omega_I)$ are equivalent $K$-homology classes. As is suggested in \cite{bdt:boundary}, a sharper estimate of Sobolev norms proves that $L^2_{a,s}(\Omega_I)$ is a Schatten-$p$ class module for $p> k$. Our proof of Theorem \ref{thm:essentially-normal} does not show that the ideal $\bar{I}$ or the quotient module $Q_I$ is a Schatten-$p$ class module, though we expect that they are and $Q_I$ is equivalent to $L^2_{a,M}$ as Schatten-$p$ class modules. 
\end{remark}

\begin{remark}\label{rmk:ee} In Theorem \ref{thm:essentially-normal}, we proved that the kernel $\ker R$ is an essentially normal $A$-module,  when $Z_I$ is the zero variety of functions $f_1, \cdots, f_M$ that are holomorphic on a neighborhood of the closed ball $\overline{\ball}^m$ and satisfy Assumption \ref{ass:principal}.  This is a variant of the ``Geometric Arveson-Douglas Conjecture" proposed by Englis and Eschmeier \cite{ee:geometric}. Englis and Eschmeier \cite{ee:geometric} proved that the kernel $\ker R$ is an essentially normal $A$-module when $Z_I$ is a homogeneous variety with the only singularity at the origin $O\in \mathbb{C}^m$. This result and our Theorem \ref{thm:essentially-normal} are independent supporting evidences for the ``Geometric Arveson-Douglas Conjecture." The two results use different methods. The main tool in \cite{ee:geometric} is the Boutet de Monvel-Guillemin theory of generalized Toeplitz operators, while the main tool in this paper is the Baum-Douglas-Taylor theory of relative $K$-homology for manifolds with boundaries. 
\end{remark}

In \cite{do-wa:quotient}, the first author and Wang proved that when $I$ is a principal ideal of $A=\mathbb{C}[z_1, \cdots, z_m]$, the quotient Hilbert module $Q_I$ is essentially normal. Let $p$ be a generator of $I$. The zero set of $p$ is a hypersurface $Z_I$ of $\complex^m$. Assumption \ref{ass:principal} in this case requires that the 1-form $dp$ is everywhere nonzero on  $\partial \Omega_I$ and $Z_I$ intersects with the sphere $\partial \mathbb{B}^m$ transversely.  

\begin{corollary}\label{cor:principal} For $m\geq 3$, when $I$ is generated by $p\in A$ and satisfies Assumption \ref{ass:principal}, the K-homology class of $Q_I$ is the fundamental class of $\partial \Omega_I$. 
\end{corollary} 

As a special example of Corollary \ref{cor:principal}, we consider the following polynomial 
\[
p_k(z_1, \cdots, z_5)= z_1^2+z_2^2+z_3^2+ z_4^3+ z_5^{6k-1} \in \mathbb{C}[z_1, \cdots, z_5],\ k\geq 1.
\]
The zero variety $Z_{p_k}$ of $p_k$ has an isolated singularity at the origin, and when $\epsilon>0$ is sufficiently small, $Z_{p_k}$ intersects with the sphere $\mathbb{S}^9_\epsilon=\partial \ball^5_\epsilon=\overline{\ball}_\epsilon^5\backslash \ball_\epsilon^5$ transversely \cite{br:exotic}, \cite{hi:survey}, where $\ball^5_\epsilon$ is the open ball of radius $\epsilon$  around the origin. Hence the conditions of Corollary \ref{cor:principal} are satisfied on $\ball^5_\epsilon$. We conclude that $Q_{I_k}$ gives the fundamental class of the boundary $\partial \Omega^\epsilon_{I_k}=Z_{p_k}\cap \partial \ball^5_\epsilon$.  The boundary $\partial \Omega^\epsilon_{I_k}$ is a topological $7$-sphere $\mathbb{S}^7$. When $k=1, \cdots, 28$, the differentiable structures on $Z_{p_k}\cap \partial \ball_\epsilon^5 $ give all the different  differentiable structures on  $\mathbb{S}^7$. Corollary \ref{cor:principal} offers a possibility to use operator algebra tools to study differentiable topology on $\mathbb{S}^7$. We plan to come back to this question in the near future. 

\section{Concluding remarks}\label{sec:rmk}

We end this article with a few remarks. 
\subsection{Assumption \ref{ass:principal} on complete intersection}\label{subsec:ass}
Our results in this article rely crucially on Assumption \ref{ass:principal} that the ideal $I$ is generated by $p_1,\cdots, p_M\in \complex[z_1, \cdots, z_m]$ with $M\leq m-2$ such that the Jacobian matrix $(\partial p_i/\partial z_j)$ is of maximal rank on the boundary $\partial \Omega_I=Z_I\cap \partial \ball^m$ and the zero variety $Z_i$ intersects $\partial \ball^m$ transversely. 

Using the concept of depth \cite[Sec.18.5]{ei:book} in algebraic geometry, we can easily show that Assumption \ref{ass:principal} implies that the ideal $I$ is radical. Let $\Sigma_I$ be the set of singular sets in $\Omega_I$. Assumption \ref{ass:principal} also implies that $\Sigma_I$ is a finite discrete set having no intersection with the boundary $\partial \Omega_I$. Furthermore, Assumption \ref{ass:principal} implies that $A/I$ is a complete intersection ring \cite[Sec.18.5]{ei:book}, which assures that \cite[Ch.3, Ex.3.5]{ha:book} any analytic function on $\Omega_I-\Sigma_I$ has an extension to $\Omega_I$.  This is the key in our proof that the restriction map $R:L^2_a(\ball^m)\to L^2_{a,M}(\Omega_I)$ is surjective. 

As is explained in Proposition \ref{cor:principal}, there are many examples satisfying Assumption \ref{ass:principal}. On the other hand, there are also many interesting examples on which Assumption \ref{ass:principal} fails to hold. Let $\langle z_1, z_2\rangle$  (and $\langle z_3, z_4\rangle$) be ideals of $\complex[z_1, \cdots, z_4]$ generated by $z_1, z_2$ (and $z_3, z_4$). Consider the ideal $I$ of $\complex[z_1, \cdots, z_4]$ as the product of $\langle z_1, z_2\rangle$ and $\langle z_3, z_4\rangle$.  $I$ is generated by $z_1z_3, z_1z_4, z_2z_3, z_2z_4$. If Assumption \ref{ass:principal} holds for $I$, $Z_I$ has dimension $0$. But the zero variety of $I$ has dimension 2. This shows that $I$ fails to satisfy Assumption \ref{ass:principal}. In general,  the ideal $I$ fails to satisfy the Cohen-Macaulay condition \cite[Sec.18.2]{ei:book}, and therefore there are analytic functions on $\Omega_I-\Sigma_I$ that cannot be extended to $\Omega_I$. Such a failure suggests that the restriction map $R:L^2_a(\ball^m)\to L^2_{a,M}(\Omega_I)$ cannot be surjective for the ideal $I=\langle z_1, z_2\rangle\langle z_3, z_4\rangle$. On the other hand, in this case the range of the restriction map $R$  still has a finite codimension, which is sufficient for us to conclude Theorem \ref{thm:essentially-normal}. This suggests to generalize our results in this article to include examples like $I=\langle z_1, z_2\rangle\langle z_3, z_4\rangle$. We will come back to this in the near future. 

\subsection{Non-radical ideals}

In this article, we have considered only radical ideals. Chen, the first author, Keshari, and Xu take up the simplest non-radical cases, the ideal $I_{\underline{\alpha}}$ generated by the monomial $\underline{z}^{\underline{\alpha}}=z_1^{\alpha_1}\cdots z_m^{\alpha_m}$ for non-negative integers $(\alpha_1, \cdots, \alpha_m)$. They show that $I_{\underline{\alpha}}$ defines a sequence of zero sets $Z_k\subseteq Z_{k-1}\subseteq\cdots Z_2\subseteq Z_1$, where $Z_i$ is a union of hyperplanes intersecting at the origin such that a hyperplane for $Z_i$ contains a hyperplane for $Z_{i+1}$ or is otherwise distinct from it. 

The ideas in this paper can be applied to study non-radical ideals. When $I$ is not radical, the kernel of the restriction map $R:L^2_a(\ball^m)\to L^2_{a,M}(\Omega_I)$ is not the closure of $I$. However, by taking the zero variety $Z_I$ in the sense of schemes in algebraic geometry, we expect to enrich the space $L^2_{a,M}(\Omega_I)$ to a bigger Hilbert space $\mathcal{L}^2_{a, M}(\Omega_I)$. With the bigger space $\mathcal{L}^2_{a,M}(\Omega_I)$, we expect to show that the restriction map $\mathcal{R}$ is surjective or has a finite codimensional range and the kernel $\ker(R)$ is the closure of $I$. With this modification, we can use the methods developed in this article to study the essentially normal property of $\overline{I}$ and $Q_I$.  

The following example might be useful in understanding this phenomenon. Let $m=2$, and $I,J$ be the principal ideals generated by $z_1$ and $z_1^2$, respectively. Then the zero variety of both $I$ and $J$ is $Z_I=Z_J=\{z_1=0\}$. And $\Omega_I=\Omega_J$ is $Z_I\cap \ball^2$. The kernel $\ker R$ of the usual restriction map $R_I: L^2_{a}(\ball^2)\to L^2_{a, 1}(\Omega_I)$ is the closure of $I$ in $L^2_a(\ball^2)$. Let $\mathcal{L}^2_{a, 1}(\Omega_J)$ be $ L^2_{a,1}(\Omega_I)\oplus L^2_{a,2}(\Omega_I)$. Define the restriction map $R_J: L^2_a(\ball^2)\to\mathcal{L}^2_{a, 1}(\Omega_J)$ to be
\[
R_J(f)=f|_{z_1=0}\oplus \frac{\partial f}{\partial z_1}|_{z_1=0}.
\] 
It is straight forward to check that the kernel $\ker R_J$ of the map $R_J$ is the closure of  $J$ in $L^2_{a}(\ball^2)$, and $R_J$ is surjective. Therefore, the results of this paper extend to show that the closure $\overline{J}$ and the quotient $L^2_a(\ball^2)/\overline{J}$ are essentially normal $A$-modules (see also \cite{gw:essential-normal1}). 
 
\section{Appendix: the closure of the ideal $I$ in $L^2_a(\ball^m)$}

\begin{center}by R. Douglas, X. Tang, K. Wang\footnote{School of Mathematical Sciences, Fudan University, Shanghai, 200433, P.R. China, Email: kwang@fudan.edu.cn}
, and G. Yu
\end{center}

\vskip 3mm

In this appendix, we study the closure of the ideal $I$ in $L^2_a(\ball^m)$. We prove that under Assumption \ref{ass:principal} if $f\in L^2_a(\ball^m)$ vanishes on $\Omega_I$, then $f$ is contained in the closure $\overline{I}$ of $I$ in $L^2_a(\ball^m)$.  Related results can be found in \cite{am:ideal} and \cite{pu:dense}. 

Let $\calo(\overline{\ball}^m)$ be the algebra of holomorphic functions on the closed ball $\overline{\ball}^m$.  $\calo(\overline{\ball}^m)$ with the natural topology is a (topological) unital commutative noetherian ring \cite{ge:rings} satisfying the Hilbert nullstellensatz, i.e. every ideal of $\calo(\overline{\ball}^m)$ is either dense, or contained in a maximal ideal $\calm$ in $\calo(\overline{\ball}^m)$ such that the $\calm$-adic topology is weaker than the topology on $\calo(\overline{\ball}^m)$. 

\begin{lemma}\label{lem:local} Let $I$ be a prime ideal of $\calo(\overline{\ball}^m)$ and $f\in \calo(\overline{\ball}^m)$.  Denote $Z_I$ to be the zero variety of the ideal $I$. If there is  a point $z_0\in Z_I\cap \ball^m$ and an open set $U$ of $z_0$ in $\ball^m$, such that $f $ belongs to the ideal $I\calo(U)$ generated by $I$ in $\calo(U)$, the ring of holomorphic functions on $U$, then $f\in I$. 
\end{lemma}

\begin{proof} Let $M_{z_0}$ be the maximal ideal of $\calo(\overline{\ball}^m)$ generated by analytic functions in $\calo(\overline{\ball}^m)$ vanishing at $z_0$. Consider the $M_{z_0}$-adic completion of $I$ in $\calo(\overline{\ball}^m)$. As $z_0\in Z_I$, $I\subseteq M_{z_0}$ and $I+M_{z_0}=M_{z_0}$. By Krull's theorem \cite[Sec. 2.1]{cg:book},  the $M_{z_0}$-adic closure of $I$ in $\calo(\overline{\ball}^m)$ is $I$.  By definition, the $M_{z_0}$-adic closure of $\calo(\overline{\ball}^m)$ is 
\[
\bigcap_{j\geq 1} [I+M_{z_0}^j]. 
\]

Consider the function $f$. We prove that $f$ is in $I+M_{z_0}^j$ for every $j\geq 1$. On $U$, as $f$ belongs to the ideal $I\calo(U)$, there are $p_1, \cdots, p_l\in I$, and $h_1,\cdots, h_l\in \calo(U)$ such that 
\[
f=\sum_{i=1}^l p_i h_i. 
\]
Choose $g_1, \cdots, g_l$ in the polynomial ring $A=\complex[z_1, \cdots, z_m]$ such that $h_i-g_i$ vanishes at $z_0$ of multiplicity $j$. Observe that 
\[
f-\sum_{i=1}^l p_ig_i=\sum_{i=1}^l p_i(h_i-g_i)
\]
vanishes at $z_0$ of multiplicity $j$. Hence, $f-\sum_{i=1}^l p_ig_i$ is contained in $M^j_{z_0}$. Therefore, we have proved that $f$ belongs to $I+M_{z_0}^j$. Varying $j$ through all natural numbers, we conclude that $f$ belongs to $I=\bigcap_{j\geq 1} [I+M_{z_0}^j]$. 
\end{proof}

\begin{lemma}\label{lem:gelca} Under Assumption \ref{ass:principal}, if $f\in \calo(\overline{\ball}^m)$ vanishes on $\Omega_I$, then $f$ belongs to the closure $\langle I\rangle$ of $I$ in $\calo(\overline{\ball}^m)$, and therefore in the closure $\overline{I}$ of $I$ in $L^2_a({\ball}^m)$. 
\end{lemma}
\begin{proof}We consider $\calo(\overline{\ball}^m)$ with the topology from $H^\infty(\ball^m)$ and the ideal $I\calo(\overline{\ball}^m)$ in $\calo(\overline{\ball}^m)$. Let $I\calo(\overline{\ball}^m)=\cap_i^l I_i$ be the  irredundant  primary decomposition of $I\calo(\overline{\ball}^m)$. By \cite[Theorem 1.3]{ge:rings}, the closure of $I\calo(\overline{\ball}^m)$ in $\calo(\overline{\ball}^m)$ is the intersection of those $I_i$ that is contained in a closed maximal ideal $M_{z_i}$. It is sufficient to prove that $f$ belongs to all those $I_i$ that is contained in a closed maximal ideal $M_{z_i}$. 

By Assumption \ref{ass:principal}, $Z_I$ is smooth near the boundary $\partial \ball^m$  and intersects with $\partial \ball^m$ transversely. Hence, if there is a point $z_i\in Z_{I_i}\cap \partial \ball^m$, $z_i$ is a smooth point of $Z_I$, and therefore a smooth point of $Z_{I_i}$. We remark that as $z_i$ is a smooth point of $Z_I$, there can not be a different $I_{i'}$ such that $Z_{I_i}$ and $Z_{I_{i'}}$ intersect at $z_i$. Since $Z_I$ intersects with the boundary $\partial \ball^m$ transversely, $Z_{I_i}$ must intersect with the boundary $\partial \ball^m$ transversely at $z_i$ too. So there must be a point $\tilde{z}_i$ in $Z_{I_i}\cap \ball^m$. Furthermore, on a sufficiently small neighborhood $U_i$ of $\tilde{z}_i$ in $\ball^m$, $f$ vanishes on $Z_{I_i}\cap U_i$, and therefore $f$ is contained in the ideal $I_i \calo(U_i)$ generated by $I_i$ in $\calo(U_i)$.  Lemma \ref{lem:local} applies to $f$ and $I_i$ with $U_i$ and $\tilde{z}_i$. Therefore, $f$ is contained in $I_i$ for all $i$ such that $I_i$ is contained in a closed maximal ideal $M_{z_i}$.  Hence, $f$ is in the closure $\langle I \rangle$ of $I$ in $\calo(\overline{\ball}^m)$, and therefore also in the closure $\overline{I}$ in $L^2_a(\ball^m)$. 
\end{proof}

Let $R$ be the restriction operator $R: L^2_a(\ball)\to L^2_{a, M}(\Omega_I)$.  
\begin{theorem}\label{thm:closure}Under Assumption \ref{ass:principal}, the closure $\overline{I}$ of the ideal $I$ in $L^2_a(\ball^m)$ is equal to $\ker R$. 
\end{theorem}
\begin{proof}
We observe that if $f$ belongs to $\calo(\overline{\ball}^m)\cap \ker R$, then Lemma \ref{lem:gelca} implies that $f$ belongs to the closure $\overline{I}$ of $I$ in $L^2_a(\ball^m)$.  In the following, we show that every function in $\ker R$ can be approximated arbitrarily closely by functions in $\calo(\overline{\ball}^m)\cap \ker R$. 

Let $f$ be a function in the kernel $\ker R$.  For $0<r< 1$, define $f_r$ to be $f_r(z):=f(rz)$. $f_r$ is a holomorphic function on the (open) ball $\ball^m_{1/r}$ of radius $1/r$ in $\complex^m$. Furthermore, we have the following properties of $f_r$ on the unit ball $\ball^m$.
\begin{enumerate}
\item 
\[
\lim_{r\to 1} f_r(z)=f(z),\qquad \forall z\in \ball^m.
\]
\item 
\[
\lim_{r\to 1} ||f_r||_{L^2_a(\ball^m)}=||f||_{L^2_a(\ball^m)}.
\]
\end{enumerate} 
From the above two properties, we conclude that $f_r$ converges to $f$ in $L^2_a(\ball^m)$ as $r\to 1$. 

Fix $\epsilon>0$. Since $f_r$ converges to $f$ in $L^2_a(\ball^m)$, there exists a number $r_1$ with $0<r_1<1$ such that $||f-f_{r_1}||_{L^2_a(\ball^m)}<\epsilon$. As $R(f)=0$, $R(f-f_{r_1})=-R(f_{r_1})\in L^2_{a,M}(\Omega_I)$.  Furthermore, as $R:L^2_a(\ball^m)\to L^2_{a,M}(\Omega_I)$ is bounded,  
\[
||R(f_{r_1})||_{L^2_{a,M}(\Omega_I)}=||R(f-f_{r_1})||_{L^2_{a,M}(\Omega_I)}\leq ||R||\cdot ||f-f_r||_{L^2_{a}(\ball^m)}=||R||\epsilon. 
\]

For $0<s<1$, consider the ball $\ball^m_{1/s}\subset \complex ^m$ of radius $1/s$ with the boundary $\partial \ball^m_{1/s}=\mathbb{S}^{2m-1}_{1/s}$. Let $\Omega^s_{I}$ be the intersection of $\ball^m_{1/s}$ with $Z_I$. Assumption \ref{ass:principal} implies that there is a number $S$ with $0<S<1$ such that the similar properties as in Assumption \ref{ass:principal} also hold on the $\Omega_I^{s}$ and $\partial \Omega_I ^s$ for all $s$ with $S<s<1$. Let $L^2_{a}(\ball^m_{1/s})$ and $L^2_{a,M}(\Omega^s_{I})$ be the corresponding (weighted) Bergman spaces. Then Theorem \ref{thm:extension} can be naturally generalized to produce an extension operator 
\[
E_s: L^2_{a,M}(\Omega_I^s)\longrightarrow L^2_{a}(\ball^m_{1/s}),
\]
and a restriction operator 
\[
R_s: L^2_{a}(\ball^m_{1/s})\longrightarrow L^2_{a,M}(\Omega_I^s),
\]
such that $R_sE_s=Id$. Furthermore, following the estimates in \cite[Lemma 3.10 and Theorem 4.1]{be:extension}, we can obtain that there is a number $S'$ with $0<S'<1$ and $M>0$ such that $||E_s||<M$ for all  $s$ satisfying $S'<s<1$.  

When $r_1<s<1$, $R_s(f_{r_1})$ is well defined in $L^2_{a,M}(\Omega_I^{s})$. As $Z_I$ intersects with $\mathbb{S}^{2m-1}_{1/s}$ transversely, the defining function $\rho_s(z)$ for $\Omega_I^s$ is continuous with respect to $s$. When $s$ goes to 1, $||R_s(f_{r_1})||_{L^2_{a,M}(\Omega_I^s)}$ converges to $||R(f_{r_1})||_{L^2_{a,M}(\Omega_I)}$. In particular, there is a number $s_1$ with $r_1<s_1<1$, such that 
\[
||R_{s_1}(f_{r_1})||_{L^2_{a,M}(\Omega_I^{s_1})}\leq 2||R(f_{r_1})||_{L^2_{a,M}(\Omega_I)}\leq 2||R||\epsilon.
\]
As the norm of $E_s$ is uniformly bounded by $M$ for $S<s<1$, we will choose $s_1$ such that $||E_{s_1}||<M$. Then $||E_{s_1}R_{s_1}(f_{r_1})||_{L^2_a(\ball^m_{1/s_1})}$ is bounded by $2M||R|| \epsilon$. 

Observe that both $f_{r_1}$ and $E_{s_1}R_{s_1}(f_{r_1})$ are holomorphic on $\overline{\ball}^m$. Define $F_\epsilon\in \calo(\overline{\ball}^m)$ by $F_\epsilon(z):=f_{r_1}(z)-E_{s_1}R_{s_1}(f_{r_1})(z)$. Then 
\[
R(F_\epsilon)=R(f_{r_1})-R(E_{s_1}R_{s_1}(f_{r_1}))=R(f_{r_1})-R(f_{r_1})=0. 
\]
Hence, $F_\epsilon$ is inside $\ker R$, and therefore in $\calo(\overline{\ball}^m)\cap \ker R$.  We have the following estimate of the norm $||f-F_\epsilon||_{L^2_a(\ball^m)}$. 
\begin{eqnarray*}
&||f-F_\epsilon||_{L^2_a(\ball^m)}&=||f-f_{r_1}+E_{s_1}R_{s_1}(f_{r_1})||_{L^2_a(\ball^m)}\leq ||f-f_{r_1}||_{L^2_a(\ball^m)}+\\
&&+||E_{s_1}R_{s_1}(f_{r_1})||_{L^2_a(\ball^m)}
\leq \epsilon+||E_{s_1}R_{s_1}(f_{r_1})||_{L^2_a(\ball^m_{1/s_1})}\leq (2M||R||+1)\epsilon.
\end{eqnarray*}
Hence we conclude that $f$ can be approximated arbitrarily closely by functions in $\calo(\overline{\ball}^m)\cap \ker R$, and therefore $f$ is in the closure $\overline{I}$ of $I$ in $L^2_a(\ball^m)$.  This shows that $\ker R$ is contained inside $\overline{I}$. 

Finally, as $\overline{I}$ is naturally contained in $\ker R$, we obtain that $\overline{I}=\ker R$. 

\end{proof}

The result in Theorem \ref{thm:closure} was stated in a more general context without Assumption \ref{ass:principal} in \cite[Theorem. 4.1 and Remark 4.4]{ps:transversality}. But Mihai Putinar and Kunyu Guo explained to us that there are mistakes in the proofs. We refer the readers to \cite{am:ideal} and \cite{pu:dense} for more results along this direction.  In general it is an open question if the closure $\overline{I}$ of the ideal $I$ in $L^2_a(\ball^m)$ is equal to the kernel of $R$ when $I$ is radical.

\end{document}